\documentclass[12pt, leqno]{amsart} 

\usepackage[mathcal]{euscript}
\usepackage{amssymb, amsfonts, xy}
\usepackage{graphicx}
\usepackage{rotating}

\usepackage{setspace}

\xyoption{all} \CompileMatrices

\begin{document}

\def\map{\mathop{\rm map}\nolimits}
\def\ker{\mathop{\rm ker}\nolimits}
\def\im{\mathop{\rm im}}
\def\precdot{\mathop{\prec\!\!\!\cdot}\nolimits}
\def\colim{\mathop{\rm colim}}
\def\lim{\mathop{\rm lim}}
\def\ohom{\mathop{\rm \overline{Hom}}\nolimits}
\def\Hom{\mathop{\rm Hom}\nolimits}
\def\Rep{\mathop{\rm Rep}\nolimits}
\def\Aut{\mathop{\rm Aut}}
\def\id{\mathop{\rm id}\nolimits}
\def\Ast{\mathop{*}}
\def\vert{\mathop{\textsf{vert}}\nolimits}
\def\edge{\mathop{\textsf{edge}}\nolimits}
\def\groups{\mathop{\mathcal Groups}\nolimits}
\def\g{\mathop{\mathcal Graphs}\nolimits}
\def\mg{\mathop{{\mathcal mm}\mbox{\rm{-}}{\mathcal Graphs}}\nolimits}
\def\ug{\mathop{{\mathcal uu}\mbox{\rm{-}}{\mathcal Graphs}}\nolimits}

\def\presup#1{{}^{#1}\hspace{-.11em}}      
\def\presub#1{{}_{#1}\hspace{-.08em}}      

\def\intertitle#1{

\medskip

{\em \noindent #1}

\smallskip
}

\newtheorem{thm}{Theorem}[section]
\newtheorem{theorem}[thm]{Theorem}
\newtheorem{lemma}[thm]{Lemma}
\newtheorem{corollary}[thm]{Corollary}
\newtheorem{proposition}[thm]{Proposition}
\newtheorem{example}[thm]{Example}

\theoremstyle{remark}
\newtheorem{rem}[thm]{Remark}

\makeatletter
\let\c@equation\c@thm
\makeatother
\numberwithin{equation}{section}



\newcommand{\comment}[1]{}

\newif\ifshortlabels\shortlabelstrue

\ifshortlabels
  \newcommand{\mylabel}[1]{\label{#1}}
  \newcommand{\mylabeln}[1]{\label{#1}}
\else
  \newcommand{\mylabel}[1]{(#1) \label{#1}\ \ \ \ \ \ \ \ \ \ \ \ \ \ \\}
  \newcommand{\mylabeln}[1]{(#1) \label{#1}}
\fi



\title{An ``almost'' full embedding \\ of the category of graphs \\
  into the category of groups}
\author{Adam J. Prze\'zdziecki$^1$}
\address{Warsaw University of Life Sciences - SGGW, Warsaw, Poland}
\email{adamp@mimuw.edu.pl}

\maketitle
\begin{center}
\today
\end{center}

\comment{
\begin{center}
\vspace{-0.7cm}
{\small \it Institute of Mathematics, Warsaw University, Warsaw, Poland \\
Email address: \verb+adamp@mimuw.edu.pl+ } \\
\end{center}
}

\footnotetext[1]{The author was partially supported by grant
  N N201 387034 of the Polish Ministry of Science and Higher Education.}

\begin{abstract}
  We construct a functor $F:\g\to{\mathcal Groups}$
  which is faithful and ``almost'' full,
  in the sense that every nontrivial group homomorphism
  $FX\to FY$ is a composition of an inner automorphism of $FY$ and a homomorphism of the form $Ff$, for a unique map of graphs $f:X\to Y$. When $F$ is composed with the Eilenberg-Mac Lane space construction $K(FX,1)$ we obtain an embedding of the category of graphs into the unpointed homotopy category which is full up to null-homotopic maps.

  We provide several applications of this construction
  to localizations (i.e. idempotent functors); we show that the questions:\\
  (1) Is every orthogonality class reflective?\\
  (2) Is every orthogonality class a small-orthogonality class?\\
  have the same answers in the category of groups as in the category of graphs. In other words they depend on set theory: (1) is equivalent to weak Vop\v enka's principle and (2) to Vop\v enka's principle. Additionally, the second question, considered in the homotopy category, is also equivalent to Vop\v enka's principle.

\vspace{4pt}
  {\noindent\leavevmode\hbox {\it MSC:\ }}
  {18A40; 20J15; 55P60; 18A22; 03E55} 
  \vspace{-9pt}

\end{abstract}


\section{Introduction}
\mylabel{section-introduction}

  Matumoto \cite{matumoto} proved that for any graph $\Gamma$
  there exists a group $G$ whose outer automorphism group is isomorphic
  to the group of automorphisms of $\Gamma$. His result received a considerable
  attention since every group can be realized as the group
  of automorphisms of some graph.

  The main result of this article may be viewed as a functorial version
  of the above. We construct a functor $F$ from the category of graphs
  to the category of groups which is faithful and ``almost'' full,
  in the sense that the maps
  $$F_{X,Y}:\Hom_{\g}(X,Y)\to\Hom_{\mathcal Groups}(FX,FY)$$
  induce bijections
  $$\overline{F}_{X,Y}:\Hom_{\g}(X,Y)\cup\{*\}\to\Rep(FX,FY).$$
  Here $\Rep(A,B)=\Hom_{\mathcal Groups}(A,B)/B$ where
  $B$ acts on $\Hom_{\mathcal Groups}(A,B)$ by
  conjugation and $*$ is an additional point which we send to the
  trivial element of $\Rep$. A {\em graph} is a set with a binary
  relation.

  Full and faithful functors are convenient tools that allow one to
  transfer constructions and properties between categories.
  The category of graphs is very comprehensive and well researched.
  Ad\'{a}mek and Rosick\'y proved in \cite[Theorem 2.65]{adamek-rosicky}
  that every accessible category has a full embedding into the category
  of graphs. Instead of quoting the complete definition of accessible
  categories let us mention that these contain, as full subcategories,
  ``most'' of the ``non-homotopy''
  categories: the categories of groups, fields, $R$-modules, Hilbert spaces, posets
  (i.e. partially ordered sets), simplicial sets, metrizable spaces or
  CW-spaces and continuous maps, the category of models of some first-order theory,
  and many more. In fact, under a large cardinal hypothesis that the measurable cardinals are bounded above, any concretizable category fully embeds into the category of graphs \cite[Chapter III, Corollary 4.5]{trnkova}.

  In this article we describe several applications of the functor $F$,
  constructed in Section \ref{section-construction}; the choice
  of the applications is
  strongly affected by the interests of the author.

  A {\em localization} may be defined as a functor from a category $\mathcal{C}$
  to itself that is a left adjoint to inclusion of a subcategory
  $\mathcal{D}\subseteq\mathcal{C}$; it is an idempotent functor which
  may be viewed as a projection of
  $\mathcal{C}$ onto the subcategory $\mathcal{D}$. A more common definition of localization
  can be found in Section \ref{section-orthogonal}.
  Libman \cite{libman} inspired a question of whether the values of
  localization functors
  at finite groups can have arbitrarily large cardinalities.
  For all finite simple groups such localizations were constructed by
  G\"obel, Rodr\'iguez, Shelah in \cite{grs}, \cite{gs}, and for some such groups by
  the author in \cite{ap}. In Section \ref{section-large-localizations}
  we see that the functor $F$ immediately produces yet another such
  construction.

  This article was motivated by another application.
  Ad\'{a}mek and Rosick\'y proved in \cite[Chapter 6]{adamek-rosicky}
  that large cardinal axioms called Vop\v enka's
  principle and weak Vop\v enka's principle
  (both formulated in the category of graphs) have many implications
  related to localizations and the structure of accessible categories. These axioms are
  believed to be consistent with the standard set theory ZFC while
  their negations are known to be consistent with ZFC. Casacuberta,
  Scevenels and Smith \cite{casacuberta-advances} extended
  some of these implications to the homotopy category.
  In Section \ref{section-homotopy} we see that a functor
  which sends a graph $\Gamma$ to the Eilenberg-Mac Lane
  space $K(F\Gamma,1)$ is,
  up to null-homotopic maps, a full embedding of the
  category of graphs into the (unpointed) homotopy category.
  We strengthen the results of
  \cite{casacuberta-advances} by showing that Vop\v enka's principle
  is actually equivalent to its formulation in the homotopy category:
  every orthogonality class in the homotopy category is a small-orthogonality
  class in the homotopy category
  (i.e. it is associated with an $f$-localization of Bousfield
  and Dror Farjoun \cite{farjoun})
  if and only if this is the case in the category
  of graphs.

  On the other hand, it was hoped that some consequences of Vop\v enka's
  principles in the category of groups might be provable in ZFC.
  Casacuberta and Scevenels \cite{casacuberta} hint that this
  might be the case for a ``long standing open question in
  categorical group theory''
  that asks if every orthogonality class $\mathcal{D}$,
  in the category of groups, is reflective -- that is, if the
  inclusion functor $\mathcal{D}\to{\mathcal Groups}$ has a left
  adjoint. In Section \ref{section-orthogonal} we find that this
  question is actually equivalent to weak Vop\v enka's principle.

  The work presented in this paper has begun during the author's
  visit to Centre de Recerca Mathem\`{a}tica, Bellaterra, at
  the inspiration of Carles Casacuberta.

\section{Definitions}

A {\em graph} $\Gamma$ is a set of {\em vertices}, $\vert\Gamma$,
together with a set of {\em edges}, which is a binary relation
$\edge\Gamma\subseteq \vert\Gamma\times\vert\Gamma$. A morphism
$\Gamma\to\Delta$ between graphs is an edge preserving
function $\vert\Gamma\to\vert\Delta$. The category of graphs is
denoted $\g$.

An {\em m-graph} (m for multi-edge) is a category $\Gamma$ whose
objects form a disjoint union of a set of vertices, $\vert\Gamma$,
and a set of edges, $\edge\Gamma$. Each nonidentity morphism of an
m-graph $\Gamma$ has its source in $\edge\Gamma$ and its target in
$\vert\Gamma$. Each edge $e\in\edge\Gamma$ is a source of two
nonidentity morphisms: one labelled $\iota_e$ whose target is the
{\em initial vertex} of $e$, and the other labelled $\tau_e$ whose
target is the {\em terminal vertex} of $e$. Morphisms between
m-graphs are functors
that preserve the edges, the vertices and the labelling:
$f(\iota_e)=\iota_{f(e)}$ and $f(\tau_e)=\tau_{f(e)}$. The
category of m-graphs is denoted $\mg$.

A {\em u-graph} (u for undirected-edge) is an m-graph without the
labelling of morphisms. The category of u-graphs is denoted $\ug$.

A u-graph is usually visualized as in \eqref{equation-circle} where
the nonidentity morphisms are represented by incidence between
edges (intervals) and vertices (small circles). A graph
or an m-graph is similarly visualized, with arrows on its edges.

We have an obvious full and faithful inclusion functor
$I:\g\to\mg$ which has a left adjoint (the edge collapsing functor
$J:\mg\to\g$), that is,
$$\Hom_{\g}(J\Gamma,\Delta)\cong\Hom_{\mg}(\Gamma,I\Delta)$$
where $\Gamma$ is in $\mg$ and $\Delta$ is in $\g$.

A {\em graph of groups} is a functor $G:\Gamma\to\groups$ where
$\Gamma$ is a u-graph and for each morphism $i$ in $\Gamma$, $G(i)$
is a monomorphism. $\Gamma$ is called the underlying u-graph of
$G$.

{\em Convention.} If $G:\Gamma\to\groups$ is a graph of groups and
$a$, $b$ are objects in $\Gamma$, we consider the values of $G$ on
$a$ and $b$, that is, $G_a$ and $G_b$, to be
different whenever $a$ and $b$ are different, and $G$ takes morphisms
to inclusions. In short, we treat $G$ as the image of an inclusion
of $\Gamma$ into $\groups$ all of whose morphisms are inclusions.
The objects of $G$ are called the {\em edge} and the {\em vertex
groups}.

A {\em tree} (a {\em tree of groups}) is a connected u-graph
(graph of groups) without circuits, that is, closed paths without
backtracking.

If $G$ is a group, $g\in G$ and $A\subseteq G$ then $\presup{g}A$
denotes $gAg^{-1}$.

\section{Bass-Serre theory}

In this section we collect facts concerning groups acting on trees,
which will be used later. The key reference is \cite{serre}. The symbol $*_AG_i$
denotes the {\em amalgam} of groups $G_i$ along the common subgroup $A$, and
$\colim G$ denotes the {\em colimit} of a graph of groups $G$.

\begin{lemma}
\mylabel{lemma-bs-injective}
  Let $H_1\subseteq G_1$ and $H_2\subseteq G_2$ and $A$ be a
  common subgroup of $G_1$ and $G_2$. If $H_1\cap A=B=H_2\cap A$
  then the homomorphism $h:H_1*_BH_2\to G_1*_AG_2$ induced by the
  inclusions is injective.
\end{lemma}
\begin{proof}
  See \cite[\S1.3, Proposition 3]{serre}.
\end{proof}

As a consequence we obtain

\begin{lemma}
\mylabel{lemma-bs-free-product}
  Let $G$ be a graph of groups consisting of one central vertex
  group $C$ and vertex groups $B_i$, $i\in I$, attached to $C$
  along edge groups $A_i$, $i\in I$:
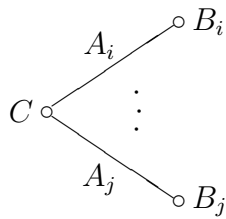
  \begin{center}
  \begin{picture}(60,80)
  \put(0,40){
    \put(0,0){\circle{4}}
    \put(-5,0){\makebox(0,0)[r]{$C$}}
    \put(34,7){\makebox(0,0){$\cdot$}}
    \put(35,0){\makebox(0,0){$\cdot$}}
    \put(34,-7){\makebox(0,0){$\cdot$}}
    \put(2,2){\line(3,2){46}}
    \put(20,20){\makebox(0,0)[b]{$A_i$}}
    \put(50,34){\circle{4}}
    \put(55,34){\makebox(0,0)[l]{$B_i$}}
    \put(2,-2){\line(3,-2){46}}
    \put(20,-20){\makebox(0,0)[t]{$A_j$}}
    \put(50,-34){\circle{4}}
    \put(55,-34){\makebox(0,0)[l]{$B_j$}}
  }
  \end{picture}
  \end{center}
  If $H_i\subseteq B_i$ are subgroups such that $H_i\cap A_i$ is
  trivial for $i\in I$ then the homomorphism
  $h:*_{i\in I}H_i\to\colim G$ induced by the inclusions is
  injective and its image trivially intersects $C$.
\end{lemma}

\begin{proof}
  We identify $I$ with an ordinal and proceed by induction. The
  case when $I$ is a singleton is obvious, as is the case when
  $I$ is a limit ordinal and the result is established for all
  $I_0<I$. Suppose that $I=I_0\cup\{i_0\}$ and the result is
  established for $I_0$. Let $G_0$ be the graph of groups obtained
  from $G$ by deleting $B_{i_0}$ and $A_{i_0}$. We have
  $$ \colim G = B_{i_0}*_{A_{i_0}}\colim G_0.$$
  By the inductive assumption, $h$ is injective on $*_{i\in I_0}H_i$
  and $h(*_{i\in I_0}H_i)\cap C$ is trivial, and therefore Lemma
  \ref{lemma-bs-injective} implies the result for $I$.
\end{proof}

The most powerful element of the Bass-Serre theory is the
following.

\begin{theorem}[{\cite[\S4.5, Theorem 9]{serre}}]
\mylabel{theorem-serre-tree}
  Let $G$ be a tree of groups and $T$ the underlying u-graph.
  There exists a u-graph $X$ containing $T$ and an action of
  $G_T=\colim G$ on $X$ which is characterized (up to isomorphism) by
  the following properties:
  \begin{itemize}
    \item[(a)] $T$ is the fundamental domain for $X$ mod $G_T$ and
    \item[(b)] for any $v$ in $\vert T$ (resp. $e$ in $\edge T$)
    the stabilizer of $v$ (resp. $e$) in $G_T$ is $G_v$ (resp. $G_e$).
  \end{itemize}
  Moreover, $X$ is a tree.
\end{theorem}

As a corollary of Theorem \ref{theorem-serre-tree} we immediately
obtain:

\rem \mylabel{remark-bs-t} Let $X$ and $G$ be as above.
\begin{itemize}
  \item[(a)] Each vertex group of $G$ is a subgroup of $\colim G$.
  \item[(b)] The stabilizers of the vertices and edges of $X$ are
  respectively the $\colim G$ conjugates of the vertex and edge
  groups of $G$.
  \item[(c)] If a subgroup $H$ of $\colim G$ stabilizes two
  vertices $v$ and $w$ in $X$ then it stabilizes the shortest path
  from $v$ to $w$ and therefore $H$ is contained in all the vertex
  and edge stabilizers of this path.
  \item[(d)] For any edge
  \begin{center}
  \begin{picture}(80,40)
  \put(0,23){
    \put(0,0){\circle{4}}
    \put(0,0){\put(0,-6){\makebox(0,0)[t]{$v$}}}
    \put(2,0){\line(1,0){56}}
    \put(30,0){\put(0,-4){\makebox(0,0)[t]{$e$}}}
    \put(60,0){\circle{4}}
    \put(60,0){\put(0,-6){\makebox(0,0)[t]{$w$}}}
  }
  \end{picture}
  \end{center}
  in $G$ we have $G_v\cap G_w=G_e$ in $\colim G$.
\end{itemize}

\begin{lemma}
\mylabel{lemma-bs-finite-subgroup}
  If $G$ is a tree of groups and $H\subseteq\colim G$ is a finite
  subgroup then $H$ is conjugate in $\colim G$ to a subgroup of some vertex
  group $G_v$.
\end{lemma}

\comment{
\begin{lemma}
\mylabel{lemma-bs-normalizer}
  In the setting of Lemma \ref{lemma-bs-injective}, if
  $H_i=N_{G_i}(H_i)$ for $i=1,2$ then $h(H_1*_BH_2)$ is equal to
  its normalizer in $G_1*_BG_2$.
\end{lemma}
}

\section{Construction of the functor $F$}
\mylabel{section-construction}
We start with the following graph of groups, where some edge to
vertex incidences are labelled with $c$:

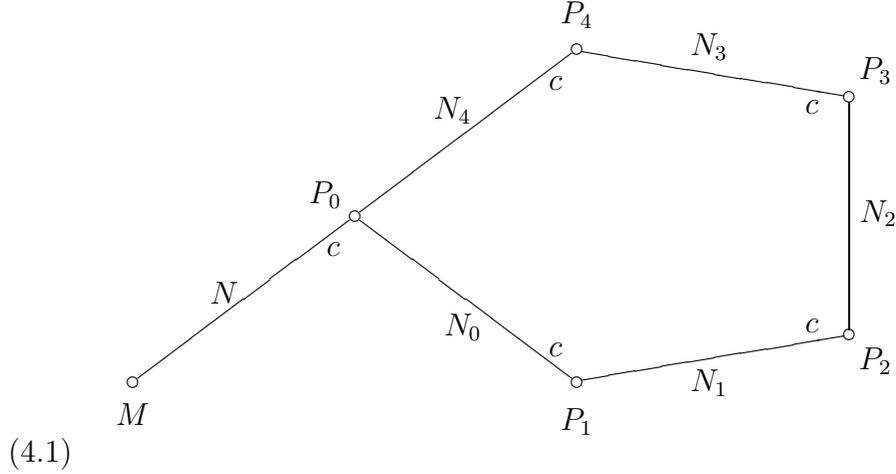
\begin{center}
\begin{equation}
\mylabeln{equation-circle}
\begin{picture}(290,190)
\put(0,30){
  \put(0,0){\circle{4}}
  \put(0,0){\put(0,-16){\makebox(0,0)[b]{$M$}}}
  \put(0,0){\put(40,30){\makebox(0,8)[r]{$N$}}}
  \put(2,1.5){\line(4,3){80}}

  \put(84,63){\circle{4}}
  \put(84,63){\put(-5,8){\makebox(0,0)[r]{$P_0$}}}
  \put(76,53){\makebox(0,0)[t]{$c$}}
  \put(84,63){\put(56,-39){\makebox(-8,-6)[r]{$N_0$}}}
  \put(86,61.5){\line(4,-3){80}}
  \put(84,63){\put(44,40){\makebox(0,0)[r]{$N_4$}}}
  \put(86,64.5){\line(4,3){80}}

  \put(168,0){\circle{4}}
  \put(168,0){\put(0,-20){\makebox(0,0)[b]{$P_1$}}}
  \put(160,10){\makebox(0,0)[b]{$c$}}
  \put(168,0){\put(50,5){\makebox(0,0)[t]{$N_1$}}}
  \put(170,1){\line(6,1){99}}

  \put(168,126){\circle{4}}
  \put(168,126){\put(0,8){\makebox(0,0)[b]{$P_4$}}}
  \put(160,116){\makebox(0,0)[t]{$c$}}
  \put(168,126){\put(50,-4){\makebox(0,0)[b]{$N_3$}}}
  \put(170,125){\line(6,-1){99}}

  \put(271,18){\circle{4}}
  \put(271,18){\put(0,-10){\makebox(20,0){$P_2$}}}
  \put(257,19){\makebox(0,0)[b]{$c$}}
  \put(271,108){\circle{4}}
  \put(271,108){\put(0,10){\makebox(20,0){$P_3$}}}
  \put(257,107){\makebox(0,0)[t]{$c$}}
  \put(271,64){\put(4,0){\makebox(0,0)[l]{$N_2$}}}
  \put(271,20){\line(0,1){86}}
}
\end{picture}
\end{equation}
\end{center}

\def\condsout{C1}
\def\condmtop{C2}
\def\condclosed{C3}
\def\condtwoedges{C4}
\def\condedgetrivial{C5}
\def\condverttrivial{C6}
\def\condfold{C7}
\def\condfactor{C8}

We assume the following conditions:
\begin{itemize}
  \item[\condsout] $M$ is finite, centerless and any homomorphism
  $f:M\to M$ is either trivial or an inner automorphism.
  \item[\condmtop] $M$ admits no nontrivial homomorphisms to $P_i$ for
  $i=0,1,\ldots,4$.
  \item[\condclosed] If an inclusion $A\subseteq B$ in
  \eqref{equation-circle} is labelled $c$
  and $f:B\to B$ is a homomorphism which is the identity on $A$
  then $f$ is the identity.
  \item[\condtwoedges] If $A_1$ and $A_2$ are edge groups ($A_1\neq N_2$) adjacent to
  the common vertex group $B$ then $A_1$ is not conjugate in $B$
  to a subgroup of $A_2$. If $A_1=A_2$ we require that
  $N_B(A_1)=A_1$.
  \item[\condedgetrivial] $N_1\cap N_2$ and $N_2\cap N_3$ are trivial.
  \item[\condverttrivial] $N_1\cap N_0$ and $N_3\cap N_4$ are trivial.
  \item[\condfold] If $A\supseteq C \subseteq B$ is an edge in
  \eqref{equation-circle} and $C\subseteq B$ is labelled $c$ then no homomorphism
  $f:B\to A$ is the identity on $C$.
  \item[\condfactor] If an inclusion $A\subseteq B$ in
  \eqref{equation-circle} is labelled $c$ and $K\subseteq B$ is a
  normal subgroup which contains $A$ then $K=B$.
\end{itemize}

\begin{lemma}
\mylabel{lemma-exists}
  There exists a graph of groups \eqref{equation-circle} satisfying
  conditions C1--C8.
\end{lemma}

\begin{proof}
  We have:
  \begin{center}
  \begin{picture}(290,180)
  \put(0,30){
    \put(0,0){\circle{4}}
    \put(0,0){\put(0,-16){\makebox(0,0)[b]{$M_{23}$}}}
    \put(0,0){\put(40,30){\makebox(0,8)[r]{$N$}}}
    \put(2,1.5){\line(4,3){80}}

    \put(84,63){\circle{4}}
    \put(84,63){\put(-5,8){\makebox(0,0)[r]{$A_{12}\oplus A_{11}$}}}
    \put(84,63){\put(51,-39){\makebox(0,-8)[r]{$N\oplus A(S_3\oplus S_8)$}}}
    \put(86,61.5){\line(4,-3){80}}
    \put(84,63){\put(40,33){\makebox(0,8)[r]{$N\oplus A(S_4\oplus S_7)$}}}
    \put(86,64.5){\line(4,3){80}}

    \put(168,0){\circle{4}}
    \put(168,0){\put(0,-16){\makebox(0,0)[b]{$A_{11}\oplus A_{12}$}}}
    \put(168,0){\put(50,14){\makebox(0,0)[b]{$A_{11}$}}}
    \put(170,1){\line(6,1){99}}

    \put(168,126){\circle{4}}
    \put(168,126){\put(0,16){\makebox(0,0)[t]{$A_{11}\oplus A_{12}$}}}
    \put(168,126){\put(50,-14){\makebox(0,0)[t]{$A_{11}$}}}
    \put(170,125){\line(6,-1){99}}

    \put(271,18){\circle{4}}
    \put(271,18){\put(0,-16){\makebox(20,0){$A_{12}$}}}
    \put(271,108){\circle{4}}
    \put(271,108){\put(0,16){\makebox(20,0){$A_{12}$}}}
    \put(271,64){\put(4,0){\makebox(0,0)[l]{$\mathbb{Z}_{12}$}}}
    \put(271,20){\line(0,1){86}}
  }
  \end{picture}
  \end{center}
  Here $M_{23}$ is the Mathieu simple group,
  $N\cong\mathbb{Z}_{11}\rtimes\mathbb{Z}_5$ is the normalizer of
  the Sylow $11$-subgroup in $M_{23}$ \cite[page 265]{gorenstein},
  $S_n$ and $A_n$ denote the $n$-th symmetric and the $n$-th alternating groups.
  $A(S_p\oplus S_q)$ is the intersection of $S_p\oplus S_q$ and $A_{12}$ in $S_{12}$.
  The inclusions are as follows:
  \begin{enumerate}
    \item  $N\subseteq A_{12}\oplus A_{11}$ is determined by any inclusions
      $N\subseteq A_{12}$ and $N\subseteq A_{11}$.
    \item  $N\oplus A(S_p\oplus S_q)\subseteq A_{12}\oplus A_{11}$ equals
      $(N\subseteq A_{12})\oplus(\mbox{natural inclusion }A(S_p\oplus S_q)
      \subseteq A_{11})$.
    \item  $N\oplus A(S_p\oplus S_q)\subseteq A_{11}\oplus A_{12}$ equals
      $(N\subseteq A_{11})\oplus(A(S_p\oplus S_q) \subseteq A_{12})$.
    \item  $A_{11}\subseteq A_{12}$ is the inclusion of a maximal subgroup.
    \item  $A_{11}\subseteq A_{11}\oplus A_{12}$ is determined by $\id_{A_{11}}$ and $A_{11}\subseteq A_{12}$.
    \item  $\mathbb{Z}_{12}\subseteq A_{12}$ is the inclusion of a transitive subgroup.
  \end{enumerate}
  We know \cite[page 265]{gorenstein} that $M_{23}$ has no outer
  automorphisms and has an element of order $23$.
  The order of $M_{23}$ is not divisible by $25$. Also all the
  automorphisms of $A_{11}$ and $A_{12}$ come from $S_{11}$
  and $S_{12}$. This and well known
  properties of symmetric groups make it straightforward to
  verify that all the conditions C1--C8 are satisfied.
\end{proof}

\intertitle{The construction of $G\Gamma$ and $F\Gamma$.}

Let $\Gamma$ be an m-graph. We construct a u-graph $A\Gamma$ as
follows. Replace each vertex $v$ in $\Gamma$ with a vertex $P_{0,v}$,
add a new vertex $M$, connect $M$ to every $P_{0,v}$ with an edge
$N_v$, and finally replace every subgraph
\begin{center}
\begin{picture}(80,40)
\put(0,23){
  \put(0,0){\circle{4}}
  \put(0,0){\put(0,-6){\makebox(0,0)[t]{$P_{0,v}$}}}
  \put(2,0){\vector(1,0){56}}
  \put(30,0){\put(0,-4){\makebox(0,0)[t]{$e$}}}
  \put(60,0){\circle{4}}
  \put(60,0){\put(0,-6){\makebox(0,0)[t]{$P_{0,w}$}}}
}
\end{picture}
\end{center}
where $e\in\Gamma$ with a subgraph
\begin{center}
\begin{equation}
\mylabeln{equation-edge}
\begin{picture}(300,30)
\put(0,30){
  \put(0,0){\circle{4}}
  \put(0,0){\put(0,-6){\makebox(0,0)[t]{$P_{0,v}$}}}
  \put(2,0){\line(1,0){56}}
  \put(30,0){\put(0,-4){\makebox(0,0)[t]{$N_{0,e}$}}}
  \put(60,0){\circle{4}}
  \put(60,0){\put(0,-6){\makebox(0,0)[t]{$P_{1,e}$}}}
  \put(62,0){\line(1,0){56}}
  \put(90,0){\put(0,-4){\makebox(0,0)[t]{$N_{1,e}$}}}
  \put(120,0){\circle{4}}
  \put(120,0){\put(0,-6){\makebox(0,0)[t]{$P_{2,e}$}}}
  \put(122,0){\line(1,0){56}}
  \put(150,0){\put(0,-4){\makebox(0,0)[t]{$N_{2,e}$}}}
  \put(180,0){\circle{4}}
  \put(180,0){\put(0,-6){\makebox(0,0)[t]{$P_{3,e}$}}}
  \put(182,0){\line(1,0){56}}
  \put(210,0){\put(0,-4){\makebox(0,0)[t]{$N_{3,e}$}}}
  \put(240,0){\circle{4}}
  \put(240,0){\put(0,-6){\makebox(0,0)[t]{$P_{4,e}$}}}
  \put(242,0){\line(1,0){56}}
  \put(270,0){\put(0,-4){\makebox(0,0)[t]{$N_{4,e}$}}}
  \put(300,0){\circle{4}}
  \put(300,0){\put(0,-6){\makebox(0,0)[t]{$P_{0,w}$}}}
}
\end{picture}
\end{equation}
\end{center}
We say that $M$, $N$, $N_i$, $P_i$ for $i=0,1,\ldots,4$ are {\em
types} of objects $M$, $N_a$, $N_{i,a}$, $P_{i,a}$ for
$i=0,1,\ldots,4$ and $a$ in $\vert\Gamma$ or $\edge\Gamma$,
respectively. We see that the resulting functor $A$ preserves
colimits of connected diagrams.

We construct a graph of groups $G\Gamma$ by taking $A\Gamma$
as the underlying u-graph
and sending each object $P$ of $A\Gamma$ to a group isomorphic
to the group in \eqref{equation-circle} labelled with the
type of $P$. We send morphisms
in $A\Gamma$ to the corresponding inclusions in
\eqref{equation-circle}. We label $c$ those inclusions in
$G\Gamma$ which correspond to similarly labelled inclusions in
\eqref{equation-circle}. The isomorphisms between the groups in
$G\Gamma$ and the groups in \eqref{equation-circle}, their
inverses and compositions are referred to as {\em standard
isomorphisms}. If $f:\Gamma\to\Gamma'$ is a
morphism of m-graphs then we define $Gf:G\Gamma\to G\Gamma'$
in the obvious way using standard isomorphisms.
We see that the resulting functor $G$, from m-graphs to graphs
of groups, preserves colimits of connected diagrams.

We define
$$F\Gamma=\colim G\Gamma,$$
in particular $F\emptyset=M$. We obtain
$Ff:F\Gamma\to F\Gamma'$ as the colimit homomorphism.

\rem\mylabel{remark-f-colimits} Since colimits commute we see that
$F$ also preserves colimits of connected diagrams.

\section{Properties of the functor $F$}

In order to apply Bass-Serre theory we need to construct $F\Gamma$
using colimits of trees of groups rather than colimits of general graphs
of groups. Let
$G_1\Gamma$ be the subgraph of groups of $G\Gamma$ consisting
of the vertices of types $M$, $P_0$, $P_1$, $P_4$ and the edges of
types $N$, $N_0$, $N_4$. Let $G_2\Gamma$ be the subgraph of
$G\Gamma$ consisting of the vertices of types $P_2$, $P_3$ and
the edges of type $N_2$. Without changing the colimit, we can make
$G_2\Gamma$ a tree of groups by adding a trivial vertex group
and connecting it to every vertex group of type $P_2$ with a
trivial edge group. Let $G_0\Gamma$ be the subdiagram of
$G\Gamma$ consisting of the edges of type $N_1$ and $N_3$. Then
$G\Gamma$ is the colimit, in the category of diagrams, of the
following:
$$G_1\Gamma\leftarrow G_0\Gamma\rightarrow G_2\Gamma.$$
Let $F_i\Gamma=\colim G_i\Gamma$ for $i=1,2,3$.
Since colimits commute, we see that $F\Gamma$ is the colimit of
$$
F_1\Gamma\leftarrow F_0\Gamma\rightarrow F_2\Gamma.
$$

It is clear that
$$F_0\Gamma=\Ast_{e\in\edge\Gamma}(N_{1,e}*N_{3,e})$$
and
$$F_2\Gamma=\Ast_{e\in\edge\Gamma}(P_{2,e}*_{N_{2,e}}P_{3,e}).$$

\begin{lemma}
\mylabel{lemma-prop-f-injective}
  The homomorphisms $F_0\Gamma\to F_i\Gamma$ for $i=1,2$ are
  injective.
\end{lemma}

\begin{proof}
  This is a consequence of Conditions {\condverttrivial} and
  {\condedgetrivial} and Lemma \ref{lemma-bs-free-product}.
\end{proof}

\begin{lemma}
\mylabel{lemma-prop-vertex-subgroups}
  The vertex groups of $G\Gamma$ map injectively into
  $F\Gamma$.
\end{lemma}

\begin{proof}
  This follows from Remark \ref{remark-bs-t}(a) and the construction
  of $F\Gamma$ by means of colimits of trees, including Lemma
  \ref{lemma-prop-f-injective}.
\end{proof}

We need an analogue of Theorem \ref{theorem-serre-tree}:

\begin{lemma}
\mylabel{lemma-prop-serre-graph}
  Let $\Gamma$ be an m-graph and $A\Gamma$ be the underlying u-graph of
  $G\Gamma$. There exists a u-graph $X$ and an action of
  $F\Gamma$ on $X$ which is characterized (up to isomorphism) by
  the following properties:
  \begin{itemize}
    \item[(a)] $A\Gamma$ is the fundamental domain for $X$ mod $F\Gamma$ and
    \item[(b)] for any $v$ in $\vert A\Gamma$ (resp. $e$ in $\edge A\Gamma$)
    the stabilizer of $v$ (resp. $e$) in $F\Gamma$ is $G\Gamma_v$ (resp. $G\Gamma_e$).
  \end{itemize}
\end{lemma}

\begin{proof}
  The proof is similar to the proof of \cite[\S4.5, Theorem
  9]{serre}: Since we know from Lemma \ref{lemma-prop-vertex-subgroups}
  that the vertex groups $G\Gamma_v$ embed into the colimit group $F\Gamma$,
  it is clear that $\vert X$ (resp. $\edge X$) is the
  disjoint union of the $F\Gamma\cdot v\cong
  F\Gamma/G\Gamma_v$ for $v\in\vert A\Gamma$ (resp. the
  $F\Gamma\cdot e\cong F\Gamma/G\Gamma_e$ for $e\in\edge
  A\Gamma$). The nonidentity morphisms are defined by means of the
  inclusions
  $G\Gamma_e\subseteq G\Gamma_{\text{target of }\iota_e}$ and
  $G\Gamma_e\subseteq G\Gamma_{\text{target of
  }\tau_e}$. This defines a graph on which the group $F\Gamma$
  acts (on the left) in the obvious way, and all the assertions of
  the lemma are immediate.
\end{proof}

\rem \mylabel{remark-prop-serre-graph-stabilizers} A subgroup of
$F\Gamma$ stabilizes a vertex or an edge of $X$ if and only if it
is conjugate in $F\Gamma$ to a subgroup of a vertex group or an edge
group of $G\Gamma$.

\begin{lemma}
\mylabel{lemma-prop-finite-x-stabilizer}
  If $H\subseteq F\Gamma$ is a finite subgroup then it
  stabilizes a vertex of $X$.
\end{lemma}

\begin{proof}
  At the beginning of this section we have presented $F\Gamma$ as the colimit
  of the following tree of groups:
  \begin{center}
  \begin{picture}(80,40)
  \put(0,23){
    \put(0,0){\circle{4}}
    \put(0,0){\put(0,-6){\makebox(0,0)[t]{$F_1\Gamma$}}}
    \put(2,0){\line(1,0){56}}
    \put(30,0){\put(0,-4){\makebox(0,0)[t]{$F_0\Gamma$}}}
    \put(60,0){\circle{4}}
    \put(60,0){\put(0,-6){\makebox(0,0)[t]{$F_2\Gamma$}}}
  }
  \end{picture}
  \end{center}
  Lemma \ref{lemma-bs-finite-subgroup}
  implies that $H$ is conjugate in $F\Gamma$ to a
  subgroup of $F_1$ or $F_2$, which again are colimits of trees of groups.
  Remark \ref{remark-prop-serre-graph-stabilizers}
  completes the proof.
\end{proof}

\begin{lemma}
\mylabel{lemma-prop-stable-path}
  Let $X$ be the u-graph as in Lemma \ref{lemma-prop-serre-graph}. If
  $N$ is a subgroup of $F\Gamma$ which stabilizes two vertices
  $P$ and $Q$ in $X$ then $N$ stabilizes some path connecting
  these vertices.
\end{lemma}

\begin{proof}
  Let $\tilde{X}$ be the tree as in Theorem
  \ref{theorem-serre-tree} for the graph of groups $G$ below:
  \begin{center}
  \begin{picture}(80,40)
  \put(0,23){
    \put(0,0){\circle{4}}
    \put(0,0){\put(0,-6){\makebox(0,0)[t]{$F_1\Gamma$}}}
    \put(2,0){\line(1,0){56}}
    \put(30,0){\put(0,-4){\makebox(0,0)[t]{$F_0\Gamma$}}}
    \put(60,0){\circle{4}}
    \put(60,0){\put(0,-6){\makebox(0,0)[t]{$F_2\Gamma$}}}
  }
  \end{picture}
  \end{center}
  Then (cf. proof of Lemma \ref{lemma-prop-serre-graph})
  $\vert\tilde{X}$ is the disjoint union of the $F\Gamma\cdot
  v\cong F\Gamma/F_i\Gamma$ for $i=1,2$, and
  $\edge\tilde{X}=F\Gamma\cdot e\cong F\Gamma/F_0\Gamma$. We
  have an $F\Gamma$-equivariant ``map'' of u-graphs
  $f:X\to\tilde{X}$ induced by the inclusions
  $G\Gamma_v\subseteq F_1\Gamma$ or $G\Gamma_v\subseteq F_2\Gamma$
  for $v\in\vert X$ and $G\Gamma_e\subseteq F_0\Gamma$ for
  $e$ in $\edge X$ and of type $N_1$ or $N_3$.
  We write ``map'' in quotation marks since
  it takes edges of type other than $N_1$ or $N_3$ to vertices --
  it is a map of diagrams but not of u-graphs.

  If $e\in \edge\tilde{X}$ then $f^{-1}(e)$ is a set of disjoint edges in $X$.
  If $v\in\vert\tilde{X}$ then $f^{-1}(v)$ is a tree isomorphic to the
  underlying tree of either $G_1\Gamma$ or $G_2\Gamma$.

  Now $N$ stabilizes $f(P)$ and $f(Q)$, and since $\tilde{X}$ is a
  tree, it stabilizes the shortest path $L$ in $\tilde{X}$,
  connecting $f(P)$ to $f(Q)$.

  If $e\in\edge L$ then the stabilizer of $e$ is
  $\presup{g}F_0\Gamma$ for some $g\in F\Gamma$, hence
  $N\subseteq\presup{g}F_0\Gamma=
  \Ast_{a\in\edge\Gamma}(\presup{g}N_{1,a}*\presup{g}N_{3,a})$.
  Since the vertex groups of $G\Gamma$ are finite, Remark
  \ref{remark-prop-serre-graph-stabilizers} implies that $N$ is finite,
  hence $N\subseteq\presup{g}N_{i,a}$ for $i=1$ or $i=3$ and some
  $a\in\edge\Gamma$. This means that $N$ stabilizes some edge in
  $f^{-1}(e)\subseteq X$.

  If $v\in\vert L$ then the stabilizer of $v$ is
  $\presup{g}F_1\Gamma$ or $\presup{g}F_2\Gamma$ for some
  $g\in F\Gamma$, hence $N\subseteq\presup{g}F_i\Gamma$ for
  $i=1$ or $i=2$ and $N$ stabilizes the tree $f^{-1}(v)\subseteq X$.
  We know that $N$ stabilizes two vertices in $f^{-1}(v)$:
  if $v$ is an inner vertex of $L$ these are
  ends of the edges in $X$, mapped by $f$ to the
  edges adjacent to $v$ in $L$, and stabilized by $N$ as seen
  above; if $v=f(P)$ or $v=f(Q)$ is an end of $L$ then one or both of these two vertices is
  $P$ or $Q$ respectively.
  Since $f^{-1}(v)$ is a tree we see that $N$ stabilizes
  the shortest path connecting these two vertices. By
  concatenating the paths and edges described above,
  we obtain the required path that connects $P$ and $Q$,
  and is stabilized by $N$.
\end{proof}

\begin{lemma}
\mylabel{lemma-prop-extensions}
  Let $A\subseteq B$ be an edge-to-vertex inclusion labelled $c$
  in \eqref{equation-circle}. Let $X$ be the u-graph as in Lemma
  \ref{lemma-prop-serre-graph} and
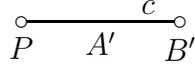
\begin{center}
\begin{picture}(80,40)
\put(0,23){
  \put(0,0){\circle{4}}
  \put(0,0){\put(0,-6){\makebox(0,0)[t]{$P$}}}
  \put(2,0){\line(1,0){56}}
  \put(48,3){\makebox(0,0)[b]{$c$}}
  \put(30,0){\put(0,-4){\makebox(0,0)[t]{$A'$}}}
  \put(60,0){\circle{4}}
  \put(60,0){\put(0,-6){\makebox(0,0)[t]{$B'$}}}
}
\end{picture}
\end{center}
  be an edge in $G\Gamma\subseteq X$ where $A'$ and $B'$ are of
  type $A$ and $B$ respectively. The standard isomorphism $f:A\to
  A'$ extends uniquely to $\overline{f}:B\to F\Gamma$, and this
  extension is the standard isomorphism onto $B'$.
\end{lemma}

\begin{proof}
  Only the uniqueness needs to be proved. Lemma
  \ref{lemma-prop-finite-x-stabilizer} implies that $f(B)$
  stabilizes a vertex $V$ of $X$. Condition {\condfold} excludes
  the case $V=P$. Lemma \ref{lemma-prop-stable-path} implies that
  $A'$ stabilizes some path connecting $V$ to $P$. If $V\neq B'$
  then $A'$ stabilizes two different edges adjacent to $P$ or to
  $B'$. This is excluded by Condition {\condtwoedges} as the
  stabilizers of edges in $X$ adjacent to a vertex $W$ in
  $G\Gamma$ are the $W$-conjugates of edges in $G\Gamma$ adjacent
  to $W$. We are left with $V=B'$, that is, $f(B)\subseteq B'$, and
  Condition {\condclosed} completes the proof.
\end{proof}

\begin{lemma}
\mylabel{lemma-prop-homomorphism-induced}
  Let $\Gamma$ and $\Delta$ be m-graphs. If $h:F\Gamma\to
  F\Delta$ is a homomorphism which restricts to the identity on $M=F\emptyset$
  then there exists a unique $f:\Gamma\to\Delta$ such that
  $h=Ff$.
\end{lemma}

\begin{proof}
  Lemma \ref{lemma-prop-extensions}, applied to $N\subseteq P_0$
  in \eqref{equation-circle}, implies that for any vertex
  $v$ in $\Gamma$ there exists a vertex $w$ in $\Delta$ such that $h$ takes
  $P_{0,v}$ in $G\Gamma$ to $P_{0,w}$ in $G\Delta$ via a
  standard isomorphism. This allows us to define $f(v)=w$. Lemma
  \ref{lemma-prop-extensions}, applied to the remaining inclusions,
  labelled $c$ in \eqref{equation-circle},
  implies that for any edge
  $e=(v_1,v_2)$ in $\Gamma$ there exist edges $e'=(f(v_1),w_2)$ and
  $e''=(w_1,f(v_2))$ in $\Delta$ such that $h$ takes, via standard
  isomorphisms, the ``half edge subgraphs'' of $G\Gamma$ to the
  ``half edge subgraphs'' of $G\Delta$ as indicated below:
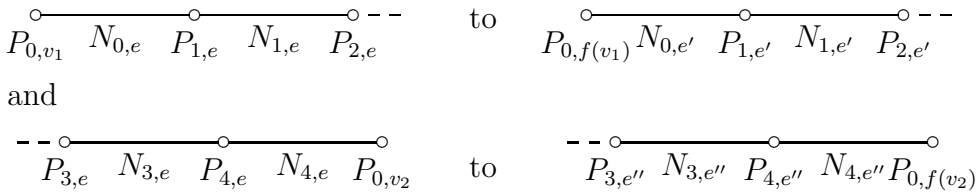
\begin{center}
\begin{picture}(150,12)
\put(-10,5){
  \put(0,0){\circle{4}}
  \put(0,0){\put(0,-6){\makebox(0,0)[t]{$P_{0,v_1}$}}}
  \put(2,0){\line(1,0){56}}
  \put(30,0){\put(0,-4){\makebox(0,0)[t]{$N_{0,e}$}}}
  \put(60,0){\circle{4}}
  \put(60,0){\put(0,-6){\makebox(0,0)[t]{$P_{1,e}$}}}
  \put(62,0){\line(1,0){56}}
  \put(90,0){\put(0,-4){\makebox(0,0)[t]{$N_{1,e}$}}}
  \put(120,0){\circle{4}}
  \put(120,0){\put(0,-6){\makebox(0,0)[t]{$P_{2,e}$}}}
  \put(126,0){\line(1,0){4}}
  \put(134,0){\line(1,0){4}}
}
\end{picture}
to
\begin{picture}(150,12)
\put(30,5){
  \put(0,0){\circle{4}}
  \put(0,0){\put(0,-6){\makebox(0,0)[t]{$P_{0,f(v_1)}$}}}
  \put(2,0){\line(1,0){56}}
  \put(30,0){\put(0,-4){\makebox(0,0)[t]{$N_{0,e'}$}}}
  \put(60,0){\circle{4}}
  \put(60,0){\put(0,-6){\makebox(0,0)[t]{$P_{1,e'}$}}}
  \put(62,0){\line(1,0){56}}
  \put(90,0){\put(0,-4){\makebox(0,0)[t]{$N_{1,e'}$}}}
  \put(120,0){\circle{4}}
  \put(120,0){\put(0,-6){\makebox(0,0)[t]{$P_{2,e'}$}}}
  \put(126,0){\line(1,0){4}}
  \put(134,0){\line(1,0){4}}
}
\end{picture}
\end{center}
\raisebox{0cm}[22pt]{and}

\begin{center}
\begin{picture}(150,20)
\put(-17,13){
  \put(0,0){\line(1,0){4}}
  \put(8,0){\line(1,0){4}}
  \put(18,0){\circle{4}}
  \put(18,0){\put(0,-6){\makebox(0,0)[t]{$P_{3,e}$}}}
  \put(20,0){\line(1,0){56}}
  \put(48,0){\put(0,-4){\makebox(0,0)[t]{$N_{3,e}$}}}
  \put(78,0){\circle{4}}
  \put(78,0){\put(0,-6){\makebox(0,0)[t]{$P_{4,e}$}}}
  \put(80,0){\line(1,0){56}}
  \put(108,0){\put(0,-4){\makebox(0,0)[t]{$N_{4,e}$}}}
  \put(138,0){\circle{4}}
  \put(138,0){\put(0,-6){\makebox(0,0)[t]{$P_{0,v_2}$}}}
}
\end{picture}
to
\begin{picture}(150,20)
\put(23,13){
  \put(0,0){\line(1,0){4}}
  \put(8,0){\line(1,0){4}}
  \put(18,0){\circle{4}}
  \put(18,0){\put(0,-6){\makebox(0,0)[t]{$P_{3,e''}$}}}
  \put(20,0){\line(1,0){56}}
  \put(48,0){\put(0,-4){\makebox(0,0)[t]{$N_{3,e''}$}}}
  \put(78,0){\circle{4}}
  \put(78,0){\put(0,-6){\makebox(0,0)[t]{$P_{4,e''}$}}}
  \put(80,0){\line(1,0){56}}
  \put(108,0){\put(0,-4){\makebox(0,0)[t]{$N_{4,e''}$}}}
  \put(138,0){\circle{4}}
  \put(138,0){\put(0,-6){\makebox(0,0)[t]{$P_{0,f(v_2)}$}}}
}
\end{picture}
\end{center}
  If $e'\neq e''$ then $P_{2,e}\cap P_{3,e}=N_{2,e}$ in
  $G\Gamma$ goes to $P_{2,e'}\cap P_{3,e''}$ which is trivial,
  and we have a contradiction. Thus $e'=e''$ and $f$
  preserves the edges.
\end{proof}

\begin{lemma}
\mylabel{lemma-prop-subgraph-subgroup}
  If $\Gamma_0$ is a sub-m-graph of $\Gamma$ then $F\Gamma_0$ is a
  subgroup of $F\Gamma$.
\end{lemma}

\begin{proof}
  It is clear that $F_i\Gamma_0$ is a
  free factor of $F_i\Gamma$ for $i=0$ and $i=2$. It is also
  clear that $G_1\Gamma_0$ is a subtree of groups of
  $G_1\Gamma$; hence, inductively applying Lemma \ref{lemma-bs-injective}
  we see that
  $F_1\Gamma_0$ is a subgroup of
  $F_1\Gamma$. We complete the proof by applying Lemma
  \ref{lemma-bs-injective} to the inclusions
  $F_i\Gamma_0\subseteq F_i\Gamma$ for $i=1,2$.
\end{proof}

\begin{lemma}
\mylabel{lemma-prop-finite-subgraph}
  Let $\Gamma$ be an m-graph. For any $g\in F\Gamma$ there
  exists a finite subgraph $\Gamma_0\subseteq\Gamma$ such that
  $g\in F\Gamma_0$.
\end{lemma}

\begin{proof}
  This is clear since $F\Gamma$ is generated by the vertex
  groups of $G\Gamma$ and each of those comes from a single
  vertex or edge in $\Gamma$.
\end{proof}

\begin{lemma}
\mylabel{lemma-prop-m}
  Let $\Gamma$ be an m-graph. For any nontrivial homomorphism
  $f:M\to F\Gamma$ there exists an inner automorphism $c_g$ of
  $F\Gamma$ such that the composition $c_gf$ is the identity on
  $M$.
\end{lemma}

\begin{proof}
  Lemma \ref{lemma-prop-finite-x-stabilizer} and Remark
  \ref{remark-prop-serre-graph-stabilizers} imply that $f(M)$ is
  conjugate in $F\Gamma$ to a subgroup of a vertex group $V$ in
  $G\Gamma$. Condition {\condmtop} and the construction of
  $G\Gamma$ imply that $V=M$, thus $c_gf(M)\subseteq M$
  for some $g$ in $F\Gamma$.
  Condition {\condsout} completes the proof.
\end{proof}

\begin{lemma}
\mylabel{lemma-prop-trivial}
  If $\Gamma$ is an m-graph, $A$ is a group and $f:F\Gamma\to A$ is
  a homomorphism which is trivial on $M$ then $f$ is trivial.
\end{lemma}

\begin{proof}
  The result follows from Condition {\condfactor} since $F\Gamma$
  is generated by the vertex groups connected to $M$ by paths
  whose edges are labelled $c$ as in \eqref{equation-circle}.
\end{proof}

If $A$ and $B$ are groups then we define $\Rep(A,B)=\Hom(A,B)/B$,
that is, we identify two homomorphisms $f,h:A\to B$ if there exists
an inner automorphism $c_g$ of $B$ such that $f=c_gh$. The set
$\Rep(A,B)$ contains a {\em trivial} element corresponding to the
trivial homomorphism.

\begin{theorem}
\mylabel{theorem-main}
  For all m-graphs $\Gamma$, $\Delta$ the composition
  $$\Hom_{\mg}(\Gamma,\Delta)\cup\{*\}\to\Hom_{\groups}(F\Gamma,F\Delta)
  \to\Rep(F\Gamma,F\Delta),$$
  where $*$ is sent to the trivial homomorphism, is bijective.
  The isomorphism is functorial in $\Gamma$ and $\Delta$.
\end{theorem}

\begin{proof}
  This is immediate from Lemmas
  \ref{lemma-prop-trivial}, \ref{lemma-prop-m} and
  \ref{lemma-prop-homomorphism-induced}.
\end{proof}

Let $\ohom(A,B)$ denote the set of nontrivial homomorphisms from
$A$ to $B$.

\rem \mylabel{remark-prop-hom-functor} $\ohom(F\Gamma,F\Delta)$ is
functorial in $\Gamma$ and $\Delta$ since $\Hom(F\Gamma,F\Delta)$
is and Lemmas \ref{lemma-prop-m} and \ref{lemma-prop-trivial}
imply that if $f:F\Gamma\to F\Delta$ and $h:F\Delta\to F\Phi$ are
nontrivial homomorphisms then $hf$ is also nontrivial.

\rem \mylabel{remark-prop-decomposition}
Note that $\Hom(\emptyset,\Delta)=\Hom_{\g}(\emptyset,\Delta)$ is
a point.
Lemmas \ref{lemma-prop-m}
and \ref{lemma-prop-homomorphism-induced} imply that for every
$f:\Hom(\emptyset,\Delta)\to\ohom(F\emptyset,F\Delta)$ we have a
pullback diagram:
$$
\xymatrix{
 \Hom(\Gamma,\Delta) \ar[r]\ar[d] & \ohom(F\Gamma,F\Delta) \ar[d]
 \\
 \Hom(\emptyset,\Delta) \ar[r]^(0.45)f &
 \ohom(F\emptyset,F\Delta)
}
$$
That is,
$$
  \ohom(F\Gamma,F\Delta)\cong\ohom(F\emptyset,F\Delta)\times\Hom(\Gamma,\Delta).
$$
The following theorem puts together Remarks
\ref{remark-prop-hom-functor} and \ref{remark-prop-decomposition}.

\begin{theorem}
\mylabel{theorem-main-2}
  For m-graphs $\Gamma$ and $\Delta$ we have a bijection
  $$
    \Hom(F\Gamma,F\Delta)\cong\ohom(F\emptyset,F\Delta)\times\Hom(\Gamma,\Delta)\cup\{*\},
  $$
  which is functorial in $\Gamma$ and $\Delta$. The $*$
  corresponds to the trivial homomorphism. A nontrivial homomorphism $h:F\Gamma\to F\Delta$
  corresponds to a pair $h|{}_{F\emptyset}$ and $f:\Gamma\to\Delta$ such that $Ff=h$.
\end{theorem}

\section{Colimits and limits}
\mylabel{section-colimits}
In this section we prove that the functor $F$ preserves directed colimits and
countably codirected limits.

We say that a poset $X$ is {\em directed (resp. countably directed)} if any finite
subset (resp. any countable subset) of $X$ has an upper bound in $X$. A poset is viewed
as a category where $a\leq b$ corresponds to a morphism $a\to b$.
A diagram (i.e. functor) $\Gamma:X\to\mathcal{C}$ and its colimit $\colim\Gamma$
are called {\em directed} if $X$ is directed. A diagram $\Gamma$ and its limit
$\lim\Gamma$ are called {\em countably codirected} if the opposite category
$X^{op}$ is countably directed.

The results of this section are stated and proved for (countably) directed
diagrams, but \cite[Theorem 1.5]{adamek-rosicky} and \cite[Remark 1.21]{adamek-rosicky}
yield immediate generalizations to the (countably) filtered case.

In this article we use Remark \ref{remark-lim-directed-colimit}
only; the remainder of this section is provided for the sake of completeness.

\intertitle{Colimits} We have noticed in Remark
\ref{remark-f-colimits} that $F:\mg\to\groups$ preserves colimits
of connected diagrams. Since the inclusion functor $I:\g\to\mg$
preserves directed colimits we obtain \rem
\mylabel{remark-lim-directed-colimit} The composition
$FI:\g\to\groups$ preserves directed colimits.

\intertitle{Limits}

The inclusion functor $I$ preserves all
limits. We investigate preservation of limits by $F$.

\begin{lemma}
\mylabel{lemma-lim-finite-intersection}
  If $\Gamma_1$ and $\Gamma_2$ are subgraphs of an m-graph
  $\Gamma$ then $F(\Gamma_1\cap\Gamma_2)=F\Gamma_1\cap
  F\Gamma_2$.
\end{lemma}

\begin{proof}
  Lemma \ref{lemma-prop-subgraph-subgroup} implies that the
  statement of the lemma makes sense. Since
  $\Gamma_1\cup\Gamma_2=\colim(\Gamma_1\supseteq\Gamma_1\cap
  \Gamma_2\subseteq\Gamma_2)$ Remark \ref{remark-f-colimits}
  implies that $F(\Gamma_1\cup\Gamma_2)=
  F\Gamma_1*_{F(\Gamma_1\cap\Gamma_2)}F\Gamma_2$
  hence the result follows from Remark \ref{remark-bs-t}(d).
\end{proof}

\begin{lemma}
\mylabel{lemma-lim-w-filtered-finite}
  If $\{\Gamma_\alpha\}_{\alpha\in A}$ is a countably codirected
  diagram of finite m-graphs then there exist $\alpha_0$ and
  $\beta$ in $A$ such that
  \begin{itemize}
    \item[(a)] the projection
    $p_0:\lim\Gamma_\alpha\to\Gamma_{\alpha_0}$ is injective,
    \item[(b)] the images of $p_0$ and
    $p^{\beta}_{\alpha_0}:\Gamma_\beta\to\Gamma_{\alpha_0}$
    coincide.
  \end{itemize}
\end{lemma}

\begin{proof}
  If $S$ is a set of objects in $\Gamma=\lim\Gamma_\alpha$ then
  for any pair $s\neq t$ in $S$ there exists $\alpha_{s,t}$ in $A$ such
  that the projection $p_{s,t}:\Gamma\to\Gamma_{\alpha_{s,t}}$ is
  injective on $\{s,t\}$. If $S$ is at most countable then there
  exists $\alpha_0$ such that each $p_{s,t}$ factors through
  $p_0:\Gamma\to\Gamma_{\alpha_0}$, hence $p_0$ is injective on
  $S$. But $\Gamma_{\alpha_0}$ is finite, hence
  $\Gamma$ is finite, and by taking $S$ to be the set of
  objects of $\Gamma$ we complete
  the proof of (a).

  If $B=\{\beta\in A\mid \beta\to\alpha_0\}$ then $\lim_{\alpha\in
  A}\Gamma_\alpha\to\lim_{\beta \in
  B}\Gamma_\beta$ is an isomorphism. Clearly $\im p_0\subseteq\im
  p^\beta_{\alpha_0}$ for $\beta\in B$. Let
  $K_\beta=(p^\beta_{\alpha_0})^{-1}(\im
  p^\beta_{\alpha_0}\setminus\im p_0)$ be viewed as a set of
  objects. If each $K_\beta$ is nonempty then, as a codirected limit of
  finite sets, $\lim K_\beta$ is nonempty, which is a contradiction
  since $\lim K_\beta\subseteq \lim\Gamma_\beta$ and
  $p_0(\lim K_\beta)\cap p_0(\lim\Gamma_\beta)=\emptyset$.
\end{proof}

\begin{lemma}
\mylabel{lemma-lim-w-filtered-closure}
  If $\{\Gamma_\alpha\}_{\alpha\in A}$ is a countably codirected
  diagram of m-graphs and $\Delta_\alpha\subseteq\Gamma_\alpha$ are
  finite subgraphs such that for all structure maps
  $p^\beta_\alpha:\Gamma_\beta\to\Gamma_\alpha$ we have
  $\Delta_\alpha\subseteq p^\beta_\alpha(\Delta_\beta)$ then
  there exist finite subgraphs
  $\overline{\Delta}_\alpha\subseteq\Gamma_\alpha$ such that
  $\Delta_\alpha\subseteq\overline{\Delta}_\alpha$ for all $\alpha$
  and $\{\overline{\Delta}_\alpha\}_{\alpha\in A}$ is a diagram, that is,
  $p^\beta_\alpha(\overline{\Delta}_\beta)\subseteq\overline{\Delta}_\alpha$.
\end{lemma}

\begin{proof}
  Define $\overline{\Delta}_\alpha$ as the union of
  $p^\beta_\alpha(\Delta_\beta)$ over all structure maps $p^\beta_\alpha$ whose
  target is $\Gamma_\alpha$. Only the finiteness of
  $\overline{\Delta}_\alpha$ needs proof. Suppose that
  $S=\{s_0,s_1,\ldots\}$ is an infinite subset of objects in
  $\overline{\Delta}_\alpha$. Then there exist
  $\alpha_0,\alpha_1,\ldots$ such that $s_i\in
  p^{\alpha_i}_\alpha(\Delta_{\alpha_i})$ for $i\in\mathbb{N}$.
  Since $\{\Gamma_\alpha\}_{\alpha\in A}$ is countably codirected
  there exists $\alpha_*$ in $A$ such that $\Gamma_{\alpha_*}$ maps to
  every $\Gamma_{\alpha_i}$ for $i\in\mathbb{N}$, hence
  $\Delta_{\alpha_i}\subseteq p^{\alpha_*}_{\alpha_i}(\Delta_{\alpha_*})$ implies
  $p^{\alpha_i}_\alpha(\Delta_{\alpha_i})\subseteq
  p^{\alpha_*}_\alpha(\Delta_{\alpha_*})$ for $i\in\mathbb{N}$,
  which is a contradiction since $\Delta_{\alpha_*}$ is finite.
\end{proof}

\begin{proposition}
\mylabel{proposition-lim-w-filtered-preserves}
  The functor $F$ constructed in Section
  \ref{section-construction} preserves countably codirected limits.
\end{proposition}

\begin{proof}
  Let $\{\Gamma_\alpha\}_{\alpha\in A}$ be a countably codirected
  diagram of m-graphs. We obtain an extended diagram
  \begin{equation}
  \mylabel{equation-lim-h}
  \xymatrix{
    {\{F\Gamma_\alpha\}}_{\alpha\in A} &
      {F\lim\Gamma_\alpha} \ar[l] \ar[dl]^h \\
    {\lim F\Gamma_\alpha} \ar[u]
  }
  \end{equation}
  where $h$ comes from the universal property of the limit. We need to
  prove that $h$ is a bijection.

  {\em Injectivity of $h$.} Let $g$ be a nonidentity element of
  $F\lim\Gamma_\alpha$. Lemma \ref{lemma-prop-finite-subgraph}
  implies the existence of a finite subgraph
  $\Gamma_0\subseteq\lim\Gamma_\alpha$ such that $g\in
  F\Gamma_0$. We look at the diagram formed by
  the images of $\Gamma_0$ in
  $\Gamma_\alpha$ for $\alpha \in A$,
  and by Lemma \ref{lemma-lim-w-filtered-finite}(a)
  we obtain $\alpha_0$ such that $\Gamma_0$ maps injectively to
  $\Gamma_{\alpha_0}$; hence Lemma
  \ref{lemma-prop-subgraph-subgroup} implies that $F\Gamma_0\to
  F\Gamma_{\alpha_0}$ is one-to-one and therefore $h(g)$ is
  nontrivial, which proves the injectivity of~$h$.

  {\em Surjectivity of $h$.} Let $g\in\lim F\Gamma_\alpha$ and let
  $g_\alpha$ be the image of $g$ in $F\Gamma_\alpha$. Let
  $\Gamma^g_\alpha\subseteq\Gamma_\alpha$ be a finite subgraph
  such that $g_\alpha\in F\Gamma^g_\alpha$ for $\alpha\in A$. Lemma
  \ref{lemma-lim-finite-intersection} implies that we may require
  $\Gamma^g_\alpha$ to be the smallest subgraph with $g_\alpha\in
  F\Gamma^g_\alpha$. The minimality implies that
  $\Gamma^g_\alpha\subseteq p^\beta_\alpha(\Gamma^g_\beta)$ for
  all structure maps $p^\beta_\alpha$, hence by Lemma
  \ref{lemma-lim-w-filtered-closure} we obtain a diagram
  $\{\overline{\Gamma}^g_\alpha\}_{\alpha\in A}$ of finite subgraphs such that
  $\Gamma^g_\alpha\subseteq\overline{\Gamma}^g_\alpha\subseteq\Gamma_\alpha$.

  Lemma \ref{lemma-lim-w-filtered-finite}(a) gives us $\alpha_0$ such
  that $p_0:\lim\overline{\Gamma}^g_\alpha\to
  \overline{\Gamma}^g_{\alpha_0}\subseteq\Gamma_{\alpha_0}$ is injective. Let $\Gamma_0$ be
  the image of $p_0$. We put the above into the following diagram,
  which is a modification of \eqref{equation-lim-h}.
  \begin{equation}
  \mylabel{equation-lim-h2}
  \xymatrix{
    {g_{\alpha_0}} \ar@{}[r]|(0.4)*+{\in} &
      {F\Gamma_0} \ar@{}[dr]|(0.44){\rotatebox{-35}{$\subseteq$}} &&&
      {F\lim\overline{\Gamma}^g_\alpha} \ar[ddlll]^{h_0}
        \ar[lll]^{Fp_0}_{\cong}  \\
    && {} \ar@{}[ul]|(0.2)*+{F\Gamma_{\alpha_0}}="name" &&
      {} \ar@{}[u]|(0.2)*+{F\lim\Gamma_\alpha}="right"  \\
    g \ar@{|->}[uu] \ar@{}[r]|(0.4)*+{\in} &
      {\lim F\overline{\Gamma}^g_\alpha}
      \ar[uu]^{q_0} &
      {\ \ \ \ \ \ \ \ }
      \ar@{}[l]|(0.2)*+{\lim F\Gamma_\alpha}="down"
      \ar@{} [uurr];"right" |{\rotatebox{-90}{$\subseteq$}}
      \ar@{} [l];"down" |*+{\subseteq}
      \ar "right";"name" |(0.585)*+{\ \ \ }
      \ar "right";"down"^h
      \ar "down";"name"|(0.535)*+{\ \ \ }
  }
  \end{equation}
  One easily deduces from Lemma \ref{lemma-lim-w-filtered-finite}(b) that
  the image of $\lim F\overline{\Gamma}_\alpha^g$ in
  $F\Gamma_{\alpha_0}$ is contained in $F\Gamma_0$,
  hence $q_0$ is well defined. $Fp_0$ is an isomorphism since $p_0$
  is an isomorphism, and therefore $q_0$ is onto.

  To complete the proof it is enough to show that $q_0$ is one-to-one.
  Suppose that $\ker q_0$ contains a nonidentity element $k$. Then
  we have a structure map $\Gamma_{\alpha_1}\to\Gamma_{\alpha_0}$
  such that $k$ is not in the kernel of
  $\lim F\overline{\Gamma}^g_\alpha\to F\Gamma_{\alpha_1}$. As
  above,
  $p_1:\lim\overline{\Gamma}^g_\alpha\to\overline{\Gamma}^g_{\alpha_1}$
  is injective and if $\Gamma_1=\im p_1$ then the image of
  $\lim F\overline{\Gamma}^g_\alpha$ in $F\Gamma_{\alpha_1}$
  is contained in $F\Gamma_1$. We obtain a modification of
  \eqref{equation-lim-h2}:
  \begin{equation}
    \xymatrix{
      {F\Gamma_0} \ar@{}[dr]|{\rotatebox{-40}{$\subseteq$}} &&
      {F\lim\overline{\Gamma}^g_\alpha} \ar[ll]_{\cong}^{Fp_0}
      \ar[dll]_(0.57){\cong}^(0.37){Fp_1=q_1h_0} \\
      {F\Gamma_1} \ar[u] \ar@{}[dr]|{\rotatebox{-40}{$\subseteq$}} &
      {F\Gamma_{\alpha_0}} \\
      & {F\Gamma_{\alpha_1}} \ar[u] \\
      {\lim F\overline{\Gamma}^g_\alpha} \ar[uu]^{q_1}
      \ar@/^2pc/[uuu]^{q_0}
    }
  \end{equation}
  and $k\in\ker q_0\setminus\ker q_1$, which is a contradiction,
  since $p_1:\lim\overline{\Gamma}^g_\alpha\to\Gamma_1$ is an isomorphism.
\end{proof}

\rem The functor $F$ does not preserve codirected limits:
Let $\Gamma_n=\mathbb{N}$ for positive integers $n$.
For $n<m$ define $p^m_n:\Gamma_m\to\Gamma_n$ as
$p^m_n(k)=\max\{0,k-(m-n)\}$. Then it is easy to see that
$\lim\Gamma_n$ is countable while $\lim F\Gamma_n$ is uncountable.

\section{Approximations of groups by graphs}
\mylabel{section-approximations}

\begin{proposition}
\mylabel{proposition-approx}
  Let $G$ be a group and $M=F\emptyset$ be as in
  Section~\ref{section-construction}.
  For every inclusion $i:M\to G$ there
  exists an m-graph $Ci$ and a diagram
  $$
  \xymatrix{
    F\emptyset \ar[r]^{\subseteq} \ar@{=}[d]&
    FCi \ar[d]_a \\
    M \ar[r]^i &
    G
  }
  $$
  such that for every m-graph $\Gamma$ and $f$ as below
  $$
  \xymatrix{
    {F\emptyset} \ar@{=}[d] \ar[r]^{\subseteq} &
      FCi \ar[d]_a \\
    M \ar@{=}[d] \ar[r]^i &
      G \\
    F\emptyset \ar[r]^{\subseteq} &
      F\Gamma \ar@/_2pc/@{-->}[uu]_{F\overline{f}} \ar[u]^f
  }
  $$
  there exists a unique $\overline{f}:\Gamma\to Ci$ for which the
  diagram above commutes.
\end{proposition}

\begin{proof}
  The construction of $Ci$ is tautological: Let $N\subseteq P_0$
  be the inclusion as in \eqref{equation-circle}. The vertices of
  $Ci$ are homomorphisms $v:P_0\to G$ such that $v|_N=i|_N$. The
  edges $v\to w$ of $Ci$ are those maps, of the graph of groups
  pictured in \eqref{equation-edge} to $G$, whose restrictions to
  $P_{0,v}$ and to $P_{0,w}$ are $v$ and $w$ respectively. The
  existence and uniqueness of $\overline{f}$ is immediate.
\end{proof}

\section{Orthogonal subcategory problem in the category of groups}
\mylabel{section-orthogonal}

In this section we apply Theorem \ref{theorem-main-2} to prove
(Proposition \ref{proposition-weak-vopenka-groups})
that if there exists an orthogonal pair in the category of graphs
which is not associated with a localization then there exists an
orthogonal pair in the category of groups which is not associated
with a localization. The premise of the implication above is consistent with the
standard set theory ZFC, in fact it is equivalent to the negation
of weak Vop\v enka's principle. We conclude this section with
Proposition \ref{proposition-vopenka-groups}. The converses of
Propositions \ref{proposition-weak-vopenka-groups} and
\ref{proposition-vopenka-groups} follow from
\cite[Theorem 6.22]{adamek-rosicky} and \cite[Corollary 6.24(iii)]{adamek-rosicky}.

In order to make the paper
self-contained we begin with a collection of definitions and
preliminary facts, most of them extracted from \cite{casacuberta}.

\intertitle{Orthogonal pairs}

Let $\mathcal{C}$ be a category (here $\groups$ or $\g$). A
morphism $f:A\to B$ is {\em orthogonal} to an object $C$
(we write $f\perp C$) if $f$
induces a bijection
\begin{equation}
\mylabel{equation-orthogonal}
\Hom_\mathcal{C}(B,C)\to\Hom_\mathcal{C}(A,C).
\end{equation}
If $\mathcal{M}$ is a class of morphisms and
$\mathcal{O}$ is a class of objects in $\mathcal{C}$ then
$\mathcal{M}^\perp=\{C\in\mathcal{C}\mid f\perp C \mbox{ for every
} f\in\mathcal{M}\}$ and $\mathcal{O}^\perp=\{f:A\to B\mid f\perp
C \mbox{ for every }C\in\mathcal{O}\}$. An {\em orthogonal pair}
$(\mathcal{S},\mathcal{D})$ consists of a class $\mathcal{S}$ of
morphisms and a class $\mathcal{D}$ of objects such that
$\mathcal{S}^\perp=\mathcal{D}$ and
$\mathcal{D}^\perp=\mathcal{S}$. If $(\mathcal{S},\mathcal{D})$ is
an orthogonal pair then $\mathcal{D}$ is called an {\em
orthogonality class}, $\mathcal{D}$ is closed under limits and
$\mathcal{S}$ is closed under colimits. If $\mathcal{M}$ is a
class of morphisms and $\mathcal{O}$ is a class of objects then
$(\mathcal{M}^{\perp\perp},\mathcal{M}^\perp)$ and
$(\mathcal{O}^\perp,\mathcal{O}^{\perp\perp})$ are orthogonal
pairs.

\intertitle{Localizations}

A {\em localization} is a functor $L:\mathcal{C}\to\mathcal{C}$
together with a natural transformation $\eta:Id\to L$ such that
$\eta_{LX}:LX\to LLX$ is an isomorphism for every $X$ and
$\eta_{LX}=L\eta_X$ for all $X$.

Every localization functor $L$ gives rise to an orthogonal pair
$(\mathcal{S},\mathcal{D})$ where $\mathcal{S}$ is the class of
morphisms $f$ such that $Lf$ is an isomorphism and $\mathcal{D}$
is the class of objects isomorphic to $LX$ for some $X$. A class
$\mathcal{D}$ is called {\em reflective} if it is part of an
orthogonal pair $(\mathcal{S},\mathcal{D})$ which is associated
with a localization.

\rem \mylabel{remark-localization-exists}
  Let $\mathcal{C}$ be a category and $(\mathcal{S},\mathcal{D})$
  an orthogonal pair in $\mathcal{C}$. If for each object $X$ in
  $\mathcal{C}$ there exists a morphism $\eta_X:X\to LX$ in
  $\mathcal{S}$ with $LX$ in $\mathcal{D}$ then the assignment
  $X\mapsto LX$ defines a localization functor associated
  with $(\mathcal{S},\mathcal{D})$; this was observed in
  \cite[1.2]{casacuberta-pairs}.

\intertitle{Weak Vop\v enka's Principle}

Weak Vop\v enka's principle is a large cardinal axiom equivalent
to the following statements:
\begin{itemize}
  \item[(WV1)] Every orthogonal pair in $\g$ is associated with a
  localization.
  \item[(WV2)] Every orthogonal pair in a locally presentable
  category ($\groups$ is such a category) is associated with a localization.
\end{itemize}
The equivalence to (WV1) is proved in \cite[Theorem 6.22]{adamek-rosicky}
and \cite[Example 6.23]{adamek-rosicky}.
The equivalence to (WV2) is proved in \cite[Example 6.25]{adamek-rosicky}
and stated in Remark that precedes it. Weak Vop\v enka's
principle is believed to be consistent with the standard set
theory (ZFC), but it is not provable in ZFC: the negation of
weak Vop\v enka's principle is consistent with ZFC. Proposition
\ref{proposition-weak-vopenka-groups} and (WV2) imply a new equivalent
formulation of weak Vop\v enka's principle:
\begin{itemize}
  \item[(WV3)] Every orthogonal pair in $\groups$ is associated with
  a localization.
\end{itemize}

More details and an interesting historical essay on Vop\v
enka's principle and its weak version can be found in
\cite{adamek-rosicky}.

\intertitle{Orthogonal subcategory problem in the category of groups}

\begin{lemma} \mylabel{lemma-f-preserves-orthogonality}
  Let $f:\Gamma\to\Phi$ be a morphism and $\Delta$ be an object in
  $\mg$. Then $f\perp\Delta$ if and only if $Ff\perp F\Delta$.
\end{lemma}

\begin{proof}
  Theorem \ref{theorem-main-2} yields
  $$
  \xymatrix{
    \Hom(F\Phi,F\Delta) \ar@{}[r]|(0.35)*+{\cong} \ar[d] &
      \ohom(F\emptyset,F\Delta)\times\Hom(\Phi,\Delta)\cup\{*\}
      \ar[d] \\
    \Hom(F\Gamma,F\Delta) \ar@{}[r]|(0.35)*+{\cong} &
      \ohom(F\emptyset,F\Delta)\times\Hom(\Gamma,\Delta)\cup\{*\}
    \\
  }
  $$
  which implies the claim (see
  \eqref{equation-orthogonal} for definition of orthogonality).
\end{proof}

\rem \mylabel{remark-sd} Throughout the remainder of this section, for a given orthogonal pair
$(\mathcal{S},\mathcal{D})$ in $\mg$ we fix an orthogonal pair
$(\overline{\mathcal{S}},\overline{\mathcal{D}})$ in $\groups$
such that
$F\mathcal{S}\subseteq\overline{\mathcal{S}}$ and
$F\mathcal{D}\subseteq\overline{\mathcal{D}}$. Such a pair
$(\overline{\mathcal{S}},\overline{\mathcal{D}})$ exists since by
Lemma \ref{lemma-f-preserves-orthogonality} we may take
$\overline{\mathcal{S}}=F\mathcal{D}^\perp$ and
$\overline{\mathcal{D}}=\overline{\mathcal{S}}^\perp$.

\begin{lemma} \mylabel{lemma-approximation-is-local}
  Let $G$ be a group in $\overline{\mathcal{D}}$ which admits an
  embedding $i:F\emptyset\to G$. If $Ci$ is the m-graph described
  in Proposition \ref{proposition-approx} then $Ci$ is in
  $\mathcal{D}$.
\end{lemma}

\begin{proof}
  Let $f:\Gamma\to\Phi$ be in $\mathcal{S}$ and $h:\Gamma\to Ci$
  be any map in $\mg$. Then the composition $F\emptyset\subseteq
  F\Gamma\to FCi\stackrel{a}{\longrightarrow}G$ equals $i$,
  and so we obtain
  $$
  \xymatrix{
    F\Gamma \ar[d]_{Ff} \ar[r]^{Fh} &
      FCi  \ar[d]^a \\
    F\Phi \ar@{-->}[r]^t \ar@{-->}[ur]^{Fs} &
      G
  }
  $$
  The unique homomorphism $t$ exists since $Ff\perp G$. The lift
  $Fs$ exists by Proposition \ref{proposition-approx}. Then
  $aFsFf=tFf=aFh$ and the uniqueness in Proposition
  \ref{proposition-approx} implies $FsFf=Fh$, hence by Theorem
  \ref{theorem-main-2} we have $sf=h$. If $s,s':\Phi\to
  Ci$ are two maps such that $sf=h=s'f$ then $aFsFf=aFs'Ff$;
  hence, as $Ff\perp G$, we have $aFs=aFs'$. Uniqueness in
  Proposition \ref{proposition-approx} yields $Fs=Fs'$, and hence by
  Theorem \ref{theorem-main-2} we obtain $s=s'$.
  Thus $f\perp Ci$ for any
  $f$ in $\mathcal{S}$ and therefore $Ci$ is in $\mathcal{D}$.
\end{proof}

\begin{lemma} \mylabel{lemma-reflective-preserved}
  If the orthogonal pair
  $(\overline{\mathcal{S}},\overline{\mathcal{D}})$ is associated
  with a localization $L$ then the pair
  $(\mathcal{S},\mathcal{D})$ is also associated with a
  localization.
\end{lemma}

\begin{proof}
  Remark \ref{remark-localization-exists} implies that it is enough
  to find for every m-graph $\Gamma$ a map
  $\eta_\Gamma:\Gamma\to\Delta$ in $\mathcal{S}$ such that
  $\Delta$ is in $\mathcal{D}$. We look at the diagram
  $$
  \xymatrix{
    & FCi \ar[d]^a \\
    **[l]{F\emptyset\subseteq F\Gamma} \ar[ur]^{Ff}
      \ar[r]^{\eta_{F\Gamma}} \ar[drr]_{Fh}&
      LF\Gamma \ar@{-->}[dr] \\
    && F\Phi
  }
  $$
  For every map $h:\Gamma\to\Phi$ with $\Phi$ in $\mathcal{D}$
  the group $F\Phi$ is in $\overline{\mathcal{D}}$, hence we
  have a factorization of $Fh$ through $\eta_{F\Gamma}$ and
  therefore a factorization of $h$ through $f:\Gamma\to Ci$.
  However, the uniqueness of the map
  $Ci\to\Phi$ under $\Gamma$ is problematic. We remedy this through an inductive
  construction. Let $\Delta_0=Ci$. If we can choose $\Phi$ in
  $\mathcal{D}$ and two different maps $g_1,g_2:\Delta_0\to\Phi$
  such that $g_1f=g_2f$ then we define $\Delta_1$ to be the limit
  of the diagram
  $$
  \xymatrix{
    \Delta_0 \ar@<0.7ex>[r]^{g_1} \ar@<-0.7ex>[r]_{g_2} & \Phi
  }
  $$
  We view $\Delta_1$ as a subgraph of $\Delta_0$, and correspondingly we
  obtain $f_1:\Gamma\to\Delta_1$. We repeat this construction
  along some ordinal $\lambda$ whose cofinality exceeds
  the cardinality of $\Delta_0$; for limit ordinals
  $\gamma<\lambda$ we define $\Delta_\gamma$ to be the limit, that
  is, the intersection, of $\{\Delta_\alpha\}_{\alpha<\gamma}$.
  Since $\{\Delta_\alpha\}$ is a strictly decreasing sequence of
  subgraphs of $\Delta_0$ it has to stabilize at some $\Delta_\beta$,
  which implies that every map $\Gamma\to\Phi$ with $\Phi$ in
  $\mathcal{D}$ factors uniquely through
  $f_\beta:\Gamma\to\Delta_\beta$, hence $f_\beta$ is in
  $\mathcal{S}$. Also $\Delta_\beta$ is in $\mathcal{D}$ since
  $Ci$ is in $\mathcal{D}$ (by Lemma
  \ref{lemma-approximation-is-local}) and
  $\mathcal{D}$ is closed under limits. Therefore
  $\eta_\Gamma=f_\beta$ is the map we were looking for.
\end{proof}

\begin{proposition}
  \mylabel{proposition-weak-vopenka-groups}
  Assuming the negation of weak Vop\v enka's principle, there
  exists an orthogonal pair in the category of groups which is not
  associated with any localization.
\end{proposition}

\begin{proof}
  The negation of (WV1) implies the
  existence of an orthogonal pair $(\mathcal{S}_0,\mathcal{D})$
  in $\g$ which is not associated with any localization. We view
  $\mathcal{S}_0$ and $\mathcal{D}$ as classes of morphisms and
  objects in $\mg$. Let $\mathcal{S}=\mathcal{D}^\perp$; since
  $\mathcal{S}_0\subseteq\mathcal{S}$ and
  $\mathcal{D}=\mathcal{S}^\perp$ we see that the orthogonal pair
  $(\mathcal{S},\mathcal{D})$ is not associated with any
  localization in $\mg$. Lemma \ref{lemma-reflective-preserved}
  implies that no pair
  $(\overline{\mathcal{S}},\overline{\mathcal{D}})$ as described
  in Remark \ref{remark-sd} is associated
  with a localization in $\groups$.
\end{proof}

\intertitle{Vop\v enka's principle and the existence of
generators}

We say that an orthogonal pair $(\mathcal{S},\mathcal{D})$ is {\em
generated} by a set of morphisms $\mathcal{S}_0$ if
$\mathcal{D}=\mathcal{S}_0^\perp$. If such a set $\mathcal{S}_0$
exists then we say that $\mathcal{D}$ is a {\em
small-orthogonality class}. A class of graphs is {\em rigid} if it
admits no morphisms except the identity morphisms (i.e. the
corresponding full subcategory is discrete). A class is {\em
large} if it has no cardinality (i.e. it is bigger than any cardinal number).

Vop\v enka's principle is another large cardinal axiom
which influences the theory of localizations. Among many
equivalent formulations of this principle we have the following ones:
\begin{itemize}
  \item[(V1)] There exists no large rigid class of graphs.
  \item[(V2)] Every orthogonality class of
  graphs is a small-orthogonality class.
  \item[(V3)] Every orthogonality class
  of objects in any locally presentable category (among those is $\groups$)
  is a small-orthogonality class.
\end{itemize}
Equivalence between these statements follows from \cite[Corollary
6.24]{adamek-rosicky} and \cite[Example 6.12]{adamek-rosicky}.

The next proposition is a nonconstructive but stronger, in terms
of the large cardinal hierarchy \cite[page 472]{kanamori}, version
of \cite[Theorem 6.3]{casacuberta-advances}. Together with (V3) it
yields another characterization of Vop\v enka's principle:
\begin{itemize}
  \item[(V4)] Every orthogonality class of
  groups is a small-orthogonality class.
\end{itemize}

\begin{proposition}\mylabel{proposition-vopenka-groups}
  Assuming the negation of Vop\v enka's principle there exists
  an orthogonal pair $(\overline{\mathcal{S}},\overline{\mathcal{D}})$ in the category
  of groups such that $\overline{\mathcal{D}}$ is not a small-orthogonality
  class.
\end{proposition}

\begin{proof}
  Negation of (V2) implies the existence of an orthogonal pair
  $(\mathcal{S},\mathcal{D})$ in $\g$ such that $\mathcal{D}$
  is not a small-orthogonality class.
  As in Remark
  \ref{remark-sd}, we have an orthogonal pair
  $(\overline{\mathcal{S}},\overline{\mathcal{D}})$ in $\groups$
  such that $F\mathcal{S}\subseteq\overline{\mathcal{S}}$ and
  $F\mathcal{D}\subseteq\overline{\mathcal{D}}$.
  Suppose that
  $\overline{\mathcal{D}}$ is a small-orthogonality class, that is, there exists
  a set $\mathcal{S}_0\subseteq\overline{\mathcal{S}}$ such that
  $\overline{\mathcal{D}}=\mathcal{S}_0^\perp$.
  Then there exists
  an uncountable cardinal $\lambda$ such that $\overline{\mathcal{D}}$ is
  closed under $\lambda$-directed colimits; it is enough that
  the cofinality of $\lambda$ is greater than all the
  cardinalities of domains and targets of maps in
  $\mathcal{S}_0$. Since
  $\mathcal{D}=F^{-1}(\overline{\mathcal{D}})$ Remark
  \ref{remark-lim-directed-colimit} implies that $\mathcal{D}$ is
  closed under $\lambda$-directed colimits. As the orthogonality
  class $\mathcal{D}$ is closed under arbitrary limits, by
  \cite[Corollary]{hebert-rosicky} it is a $\lambda$-orthogonality
  class and thus a small-orthogonality class \cite[1.35 and the following]{adamek-rosicky};
  this contradiction completes the proof.
\end{proof}

\section{Homotopy category}
\mylabel{section-homotopy}

We translate the results of the preceding section to the homotopy category
$Ho$ and to the pointed homotopy category $Ho_*$. In this section we
obtain an orthogonality preserving embedding of $\g$ into $Ho$ and a
characterization of Vop\v enka's principle in terms of the
homotopy theory. Results of \cite{casacuberta-advances} were close
to such a characterization. In this section {\em space} means
simplicial set; whenever a space $X$ is a right argument of a $\Hom$
or of a mapping space functor we assume that $X$ is fibrant.

The functor $B:\groups\to Ho_*$ which sends a group $G$ to the
Eilenberg--Mac Lane space $K(G,1)$ is full and faithful. Since
$\Hom_{Ho}(X,Y)=\Hom_{Ho_*}(X,Y)/\pi_1(Y)$ Theorem \ref{theorem-main}
implies that the
composition $BF$ followed by the forgetful functor $Ho_*\to Ho$
induces the bijections
\begin{equation}\mylabel{equation-bf-hom}
BF_{X,Y}:\Hom_{\mg}(X,Y)\cup\{*\}\to\Hom_{Ho}(BFX,BFY)
\end{equation}
where $*$ is sent to the constant map.

We say that a morphism $f:A\to B$ is {\em orthogonal} to an object
$X$ in $Ho$ if it induces an equivalence of the mapping spaces
$$\map(B,X)\to\map(A,X)$$
This notion of orthogonality is used, as in Section \ref{section-orthogonal},
to define orthogonal pairs $(\mathcal{S},\mathcal{D})$ whose
right members $\mathcal{D}$ are called orthogonality classes.
Analogously we define orthogonality in $Ho_*$ by means of
the pointed mapping spaces $\map_*(C,X)$. The fibration
$\map_*(C,X)\to\map(C,X)\to X$ for any $C$ shows that for $X$
connected we have $f\perp X$ in $Ho$ if and only if
$f\perp X$ in $Ho_*$ for any choice of base points \cite[Chapter 1, A.1]{farjoun}.

If $X$ is an
Eilenberg--Mac Lane space then $\map(A,X)$ is homotopy equivalent to
a discrete space whose underlying set is $\Hom_{Ho}(A,X)$. Thus
\eqref{equation-bf-hom} yields the following.

\begin{lemma}\mylabel{lemma-bf-preserves-orthogonality}
  Let $f:\Gamma\to\Phi$ be a morphism and $\Delta$ be an object in
  $\mg$. Then $f\perp\Delta$ if and only if $BFf\perp BF\Delta$.
\end{lemma}

The following strengthens the result of
\cite{casacuberta-advances}.

\begin{theorem}
  The following conditions are equivalent:
  \begin{itemize}
    \item[(V2)] Every orthogonality class of
      graphs is a small-orthogonality class.
    \item[({\em ho}V)] Every orthogonality class in the
      homotopy category is a small-orthogonality class.
  \end{itemize}
\end{theorem}
\begin{proof}
The implication (V2)$\implies$({\em ho}V) is \cite[Theorem
5.3]{casacuberta-advances}.

Assuming the negation of (V2), Proposition
\ref{proposition-vopenka-groups} yields an orthogonal pair
$(\mathcal{S},\mathcal{D})$ in the category of groups such that
$\mathcal{D}$ is not of the form $\mathcal{S}^\perp_0$ for any set
of morphisms $\mathcal{S}_0$. Let $f:S^2\to *$ be a map from a
$2$-sphere to a point. It is clear that a space $X$ is orthogonal to
$f$ if and only if all the connected components of $X$ are
Eilenberg--Mac Lane spaces. Thus $f\in B\mathcal{D}^\perp$ and
$B\mathcal{D}^{\perp\perp}$ is the class consisting of those spaces
all of whose connected components are homotopy equivalent to a member
of $B\mathcal{D}$.

The remainder of the proof is similar to the proof of Proposition
\ref{proposition-vopenka-groups}. If $B\mathcal{D}^{\perp\perp}$ is
a small orthogonality class then it is closed under $\lambda$-directed
homotopy colimits, for some ordinal $\lambda$ of sufficiently large
cofinality. But then $B\mathcal{D}$ is closed under $\lambda$-directed
homotopy colimits, hence $\mathcal{D}$ is closed under $\lambda$-directed
colimits, hence $\mathcal{D}$ is a small orthogonality class, which is a
contradiction.
\end{proof}

\section{Large localizations of finite groups}
\mylabel{section-large-localizations}

In this section we obtain a third construction of a class of
localizations which send a finite simple group to groups of
arbitrarily large cardinalities. Previous examples of such
localizations are described in \cite{grs}, \cite{gs} and
\cite{ap}.

Let $M$ be a group that is part of a graph of groups satisfying
conditions C1--C8 stated before Lemma \ref{lemma-exists}; we may
take $M=M_{23}$, the Mathieu group.

\begin{theorem}
  For any infinite cardinal $\kappa$ there exists a localization $L$
  in the category of groups such that $LM$ has cardinality
  $\kappa$.
\end{theorem}

\begin{proof}
  Let $F$ be the functor constructed in Section
  \ref{section-construction}. We have $M=F\emptyset$. We know
  \cite{vopenka} that for every infinite cardinal $\kappa$ there
  exists a graph $\Gamma$ of cardinality $\kappa$ such that the
  identity is the unique morphism $\Gamma\to\Gamma$. Let
  $i:\emptyset\to\Gamma$ be the inclusion of the empty set. Clearly
  $i$ is orthogonal to $\Gamma$. Let $\eta=Fi:F\emptyset\to
  F\Gamma$. Lemma \ref{lemma-f-preserves-orthogonality} implies
  that $\eta\perp F\Gamma$.
  By \cite[Lemma 2.1]{casacuberta-survey} there exists a
  localization $L$ in the category of groups such that
  $LF\emptyset=F\Gamma$, which completes the proof.
\end{proof}

\section{Closing remarks}

\comment{
The author is often asked if it is possible to construct a similar
functor $F$ so that $F\emptyset$ is a trivial group. It is
possible if the source category is $\g$, not $\mg$. For such $F$
we may have Theorem \ref{theorem-main} still true, but not
Theorem \ref{theorem-main-2}: the homomorphism $F\emptyset\to
F\Gamma$ is never orthogonal to $F\Gamma$ (cf. Lemma
\ref{lemma-f-preserves-orthogonality}), unless $\Gamma$ is empty.
}

It is intriguing to ask the following.

{\em Question:} Does there exist a faithful functor $F$ from the
category of graphs to the category of abelian groups such that
$f\perp\Gamma$ in the category of graphs
if and only if $Ff\perp F\Gamma$ in the category of abelian groups?

Some results suggest that the category of abelian groups might
be sufficiently comprehensive to allow such a functor: there exists a
considerable literature on abelian groups with
prescribed endomorphism rings (see for example
\cite[Chapter V]{krylov}, \cite[Chapter XIV]{almost-free-modules}, \cite{dugas}).
In fact the example of an orthogonality class of groups that is not a
small-orthogonality class, constructed in
\cite[Theorem 6.3]{casacuberta-advances} under the assumption of nonexistence of
measurable cardinals, consists of abelian groups.
Also there exist
arbitrarily large sets $\{A_i\}_{i\in I}$ of abelian groups such
that $\Hom(A_i,A_i)=\mathbb{Z}$ and $\Hom(A_i,A_j)=0$ for $i\neq j$
in $I$ \cite{shelah} and such that $\Hom(A_i,A_i)=A_i$ and
$\Hom(A_i,A_j)=0$ for $i\neq j$ in $I$ \cite{erings}.

\end{document}